\documentclass{amsart}
\usepackage{amsmath, amsthm, amssymb}
\usepackage[initials]{amsrefs}
\usepackage{enumerate}
\title{Canonical forests in directed families}
\author{Joseph Flenner}
\author{Vincent Guingona}
\date{\today}

\address{University of Notre Dame \\
Department of Mathematics \\ 
255 Hurley Hall \\ 
Notre Dame, IN 46556 \\
U.S.A.}

\email{jflenner@nd.edu}
\email{guingona.1@nd.edu}

\thanks{Both authors were partially supported by the NSF}

\usepackage[all]{xy}

\newtheorem{thm}{Theorem}[section]
\newtheorem{cor}[thm]{Corollary}
\newtheorem{lem}[thm]{Lemma}

\theoremstyle{remark}

\theoremstyle{definition}

\renewcommand{\phi}{\varphi}

\newcommand{\Th}{\mathrm{Th} }

\newcommand{\Lev}{\mathrm{Lev}}
\newcommand{\Sub}{\mathrm{Sub}}
\newcommand{\Ch}{\mathrm{Ch}}
\newcommand{\fM}{\mathfrak{M}}
\newcommand{\fU}{\mathfrak{U}}
\newcommand{\cS}{\mathcal{S}}
\newcommand{\cT}{\mathcal{T}}
\newcommand{\cB}{\mathcal{B}}
\newcommand{\cP}{\mathcal{P}}
\newcommand{\bbN}{\mathbb{N}}
\newcommand{\bbZ}{\mathbb{Z}}

\newcommand{\set}[2]{\left\{ #1\ \middle|\  #2\right\}}


\begin{document}

\begin{abstract}
Two uniqueness results on representations of sets constructible in a directed family of sets are given. In the unpackable case, swiss cheese decompositions are unique. In the packable case, they are not unique but admit a quasi-ordering under which the minimal decomposition is unique. Both cases lead to a one-dimensional elimination of imaginaries in VC-minimal and quasi-VC-minimal theories.
\end{abstract}

\maketitle

\section{Introduction}

In this paper, we study canonical forms for sets constructible from a directed family of sets, in the sense of Adler \cite{adl}. Every such set is realizable as a disjoint union of swiss cheeses, balls with (finitely many) holes removed. Even after imposing some nontriviality conditions on this decomposition, however, the uniqueness of such a presentation can fail. The dividing line is given by the notion of packability. Without packability, we show in Section \ref{s2} that the swiss cheese decomposition is unique. Even in the presence of packability, however, it is possible to canonically choose an `optimal' decomposition. This is described in Section \ref{s3}.

Directed families arise in logic as the building blocks of the VC-minimal theories. Introduced by Adler in 2008, these theories have garnered interest for being well-situated in the realm of model-theoretic tameness. Classical examples such as strongly minimal and o-minimal theories are VC-minimal. On the other hand, VC-minimal theories can also be seen as a natural `simplest case' among the dependent theories. Some fundamental model-theoretic machinery has already been developed, for example Cotter and Starchenko's recent analysis of forking in VC-minimal theories \cite{cs}.

A prototypical example is given by algebraically closed valued fields, from which much of the language of directed families is derived. Our primary goal is to give suitable generalizations of Holly's study of definable sets in algebraically closed valued fields \cites{hol1,hol2}, including the elimination of imaginaries in one dimension which is detailed in Section \ref{s4}. We also point out how these results can be adapted to the somewhat weaker quasi-VC-minimal setting.

Throughout the paper, we work in a directed family of sets as defined below. It should be acknowledged that this notion is not in accordance with some other uses of the term `directed', such as in category theory.

\subsection{Directed families}

For any set $\fU$ and $\cB\subseteq\cP(\fU)\setminus\left\{\emptyset\right\}$, $\cB$ is \emph{directed} if, for all $B_0,B_1\in\cB$, one of the following holds:
\begin{enumerate}[(i)]
\item
$B_0\subseteq B_1$
\item
$B_1\subseteq B_0$
\item
$B_0\cap B_1=\emptyset$.
\end{enumerate}
$\fU$ is the \emph{universe} of $\cB$. The members of $\cB$ are called \emph{balls}, and a \emph{constructible set} is a (finite) boolean combination of balls.

If $\cB$ is directed, then $\left(\cB,\subseteq\right)$ is easily seen to be a forest, that is, a union of trees whose roots are the maximal sets in $\cB$. Moreover, if $\cB$ is directed, then so is $\cB\cup\left\{\fU\right\}$. Thus we may assume that $\fU\in\cB$, as will generally be necessary.

\section{Swiss cheese and unpackability}
\label{s2}

In this section, we study the representation of constructible sets as boolean combinations of balls and the relation of these representations to unpackability. A directed family is called \emph{unpackable} if no ball is a finite union of proper sub-balls. 

Fix a directed family $\cB\subseteq\cP(\fU)\setminus\{\emptyset\}$ with $\fU\in\cB$. A \emph{swiss cheese} is a subset of $\fU$ of the form $S=A\setminus\left(B_1\cup\ldots\cup B_n\right)$, where $A$ is a ball and $B_1,\ldots,B_n\subsetneq A$ are proper sub-balls of $A$. In this expression, $n$ may be $0$ but it is usually assumed that the expression is nonredundant in the sense that for no $i\neq j$ is $B_i\subseteq B_j$. In this case, $A$ is called a \emph{wheel} of $S$ and each $B_i$ is a \emph{hole}. Note that this notion of wheels and holes is not immediately intrinsic to the set $S$, but depends on its presentation as a swiss cheese. In fact, we will show in Theorem \ref{t12} that as long as the holes are pairwise disjoint, the unique determination of wheels and holes is equivalent to unpackability.

If $X\subseteq \fU$ is constructible, a \emph{swiss cheese decomposition} of $X$ is a finite collection of swiss cheeses $S_1,\ldots,S_n$ such that
\begin{enumerate}[(i)]
\item
$S_i\cap S_j=\emptyset$ for $i\neq j$,
\item
for no $i,j$ is a wheel of $S_i$ equal to a hole of $S_j$, and
\item
$X=S_1\cup\ldots\cup S_n$.
\end{enumerate}
We again allow $n=0$, so that $\emptyset$ always has the empty swiss cheese decomposition.

\begin{lem}
\label{l11}
Every constructible subset $X\subseteq\fU$ has a swiss cheese decomposition.
\end{lem}

\begin{proof}
To begin, write $X$ as a boolean combination of balls in the form
\[
X=\bigcup\limits_{i=1}^m\bigcap\limits_{j=1}^{n_i} B_{i,j}^{e(i,j)}
\]
(where $e(i,j)\in\left\{0,1\right\}$, $B^1=B$, and $B^0=\fU\setminus B$ is the complement of $B$). Considering one of the disjuncts $\bigcap_{j=1}^{n_i}B_{i,j}^{e(i,j)}$, let $C_1,\ldots,C_r$ be those balls that appear positively (i.e. $B_{i,j}$ for $e(i,j)=1$) and $D_1,\ldots,D_s$ those that appear negatively. Due to the assumption that the universe $\fU$ is a ball, we can always assume that $r\ge 1$.

By the intersection property of balls, either $C_1\cap\ldots\cap C_r=\emptyset$ or $C_1\cap\ldots\cap C_r=C_i$ for some $i$. In the former case, the disjunct is empty and can be removed altogether. In the latter case, we write $C=C_i$ and have
\[
\bigcap\limits_{j=1}^{n_i}B_{i,j}^{e(i,j)}=C\setminus\left(D_1\cup\ldots\cup D_s\right).
\]
If this is nonempty, it is a swiss cheese.

We have thus shown that $X$ is a union of swiss cheeses. Next suppose two of these swiss cheeses have nonempty intersection, $S_1\cap S_2\neq\emptyset$. Writing
\[
S_i=A_i\setminus\left(B_{i,1}\cup\ldots\cup B_{i,r_i}\right)
\]
with $B_{i,j}\cap B_{i,k}=\emptyset$ for $j\neq k$,
$A_1\cap A_2\neq\emptyset$ implies $A_1\subseteq A_2$ or $A_2\subseteq A_1$. Say the latter. Now let $C_1,\ldots,C_s$ be a list of all those balls $C\in\left\{B_{i,j}\right\}$ such that one of
\begin{itemize}
\item
$C=B_{1,j}\subseteq A_1\setminus A_2$
\item
$C=B_{2,j}$ and $C\subseteq B_{1,k}$ for some $j\leq r_2,\ k\leq r_1$---i.e., $C$ is a hole of $S_2$ contained in a hole of $S_1$
\item
$C=B_{1,j}$ and $C\subseteq B_{2,k}$ for some $j\leq r_1,\ k\leq r_2$---i.e., $C$ is a hole of $S_1$ contained in a hole of $S_2$
\end{itemize}
holds. Then it is straightforward to verify that $S_1\cup S_2=A_1\setminus\left(C_1\cup\ldots C_s\right)$. By repeating this process as necessary, we may now assume that we have written $X$ as a \emph{disjoint} union of swiss cheeses: $X=S_1\cup\ldots\cup S_n$, $S_i\cap S_j=\emptyset$ for $i\neq j$.

Finally, suppose that the wheel of one of these is a hole of another,
\begin{align*}
S_i=&A_i\setminus\left(B_1\cup\ldots\cup B_r\right)\\
S_j=&A_j\setminus\left(A_i\cup C_1\cup\ldots\cup C_s\right).
\end{align*}
It may of course be assumed that $A_i\cap C_k=\emptyset$ for every $k\leq s$ (recalling that nonredundancy is assumed in the definition of holes). But in this case, clearly 
\[
S_i\cup S_j=A_j\setminus\left(B_1\cup\ldots\cup B_r\cup C_1\cup\ldots\cup C_s\right).
\]
Again, this process of elimination can be repeated until arriving at a swiss cheese decomposition of $X$.
\end{proof}

Swiss cheese decompositions can be canonically derived from any finite set of balls. Consider such a set $\cS$. Then $\cS$ is partially ordered by $\subseteq$, and as noted in the introduction this ordering is a (finite) forest. Define the levels $\Lev_n(\cS)$ inductively for $n\geq 0$ by:
\[
\Lev_n(\cS)=\set{B\in\cS}{\text{$B$ is $\subseteq$-maximal in $\cS\setminus\bigcup\limits_{i<n}\Lev_i(\cS)$}}.
\]
For convenience, define also $\lambda(B,\cS)=n$ if $B\in\Lev_n(\cS)$ and, for $B\in\cS$,
\[
\Sub(B,\cS)=\set{C\in\Lev_{\lambda(B,\cS)+1}(\cS)}{C\subseteq B}.
\]
Then from $\cS$ we construct the swiss cheese decomposition
\[
\Ch(\cS)=\bigcup\limits_{B\in\Lev_{2n}(\cS)}B\setminus\left(\bigcup\limits_{C\in\Sub(B,\cS)}C\right).
\]
In other words, the balls on even levels are the wheels and the holes are the wheels' immediate predecessors in $\left(\cS,\subseteq\right)$.

\begin{figure}[h]
$\xymatrix@=1pt{ & A_1\ar@{-}[dl]\ar@{-}[dr] & & & A_2\ar@{-}[d] & & & & &\\
B_1 & & B_2\ar@{-}[dl]\ar@{-}[dr] & & B_3 & \ar@{|->}[rrrr]^-\Ch & & & & \txt{\hspace{5pt}
$A_1\setminus\left(B_1\cup B_2\right)\ \cup\ C_1$\\
 \hspace{15pt}$\cup\ C_2\ \cup\ A_2\setminus B_3$
} \\
 & C_1 & & C_2 & & & & & &
}$
\caption{The swiss cheese operator}
\end{figure}
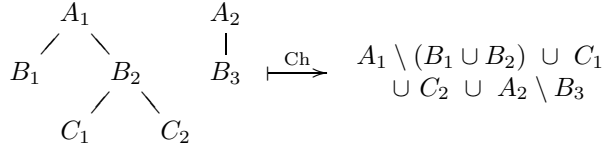

However, while constructible sets always admit a swiss cheese decomposition, and such a decomposition can always be realized as $\Ch(\cS)$ for suitable choice of $\cS$, it may occur that $\Ch(\cS)=\Ch(\cT)$ even though $\cS\neq\cT$. The notion of unpackability is central in obtaining uniqueness of the decomposition. In fact, it is equivalent. In the next theorem, (4) was first proved by Holly \cite{hol1} in the case of algebraically closed valued fields. The observation that Holly's theorem holds in this more general unpackable setting is due to Dolich (unpublished).

\begin{thm}
\label{t12}
For a directed family $\cB$ containing its universe $\fU$, the following are equivalent:
\begin{enumerate}
\item
$\cB$ is unpackable.
\item
If $A,B_1,\ldots,B_n$ are balls such that $A\subseteq\bigcup\limits_{i=1}^nB_i$, then $A\subseteq B_i$ for some $i$.
\item
If 
\[
S=A_1\setminus\left(B_{1,1}\cup\ldots\cup B_{1,m}\right)=A_2\setminus\left(B_{2,1}\cup\ldots\cup B_{2,n}\right)
\] 
is a swiss cheese and $B_{i,j}\cap B_{i,k}=\emptyset$ for $i\in\left\{1,2\right\}$, $j\neq k$, then $A_1=A_2$ and 
\[
\left\{B_{1,1},\ldots,B_{1,m}\right\}=\left\{B_{2,1},\ldots,B_{2,n}\right\}.
\]
\item
Every constructible set admits a unique swiss cheese decomposition.
\end{enumerate}
\end{thm}

\begin{proof}
We first show the equivalence of (1-3), and then show the equivalence of these three conditions with (4).

{\bf 1$\Rightarrow$2:} Suppose $A\subseteq\bigcup_{i=1}^nB_i$. We may assume that $A\cap B_i\neq\emptyset$ for each $i$. But then either $B_i\subseteq A$ or $A\subseteq B_i$. If $B_i\subsetneq A$ for every $i$, then $A$ would be a finite union of proper sub-balls, contradicting unpackability. 

{\bf 2$\Rightarrow$3:} Since $A_1\subseteq A_2\cup B_{1,1}\cup\ldots\cup B_{1,m}$, (2) gives $A_1\subseteq A_2$ or $A_1\subseteq B_{1,i}$, some $i$. But the latter cannot occur as the holes in a swiss cheese are presumed to be \emph{proper} sub-balls of the wheel. Therefore $A_1\subseteq A_2$, and by symmetry, $A_1=A_2$. It follows that
\[
B_{1,1}\cup\ldots\cup B_{1,m}=B_{2,1}\cup\ldots\cup B_{2,n}
\]
and hence $B_{1,i}\subseteq B_{2,j}$, some $j\leq n$. Similarly, $B_{2,j}\subseteq B_{1,k}$. But here we must have $k=i$, since $B_{1,i}\cap B_{1,k}\neq\emptyset$ for $i\neq k$. Thus $B_{1,i}=B_{2,j}$ and
\[
\left\{B_{1,1},\ldots,B_{1,m}\right\}\subseteq\left\{B_{2,1},\ldots,B_{2,n}\right\}.
\]
Again, (3) follows by symmetry.

{\bf 3$\Rightarrow$1:} Suppose a ball $A$ were the disjoint union of proper sub-balls $B_1,\ldots,B_n$. Necessarily $n>1$. Then the swiss cheese
$B_1=A\setminus\left(B_2\cup\ldots\cup B_n\right)$
contradicts (3).

{\bf 4$\Rightarrow$1:} This is similarly clear. If, for example, the ball $A$ were the disjoint union of proper sub-balls $B_1,\ldots,B_n$, then $X=A$ could be decomposed either as simply $A$, or as the union of the collection $B_1,\ldots,B_n$.

{\bf 1$\Rightarrow$4:} In light of Lemma \ref{l11}, it remains only to prove uniqueness. Suppose we have two decompositions, $X=\bigcup\limits_{i=1}^rS_i=\bigcup\limits_{i=1}^sT_i$ with
\begin{align*}
S_i=A_i\setminus&\left(A_{i,1}\cup\ldots\cup A_{i,m_i}\right)\\
T_i=B_i\setminus&\left(B_{i,1}\cup\ldots\cup B_{i,n_i}\right).
\end{align*}
We work by induction on $r$. If $r=0$, then $X=\emptyset$. Since no $B_i$ can be the union of its proper sub-balls $B_{i,1},\ldots,B_{i,n_i}$, no $T_i$ can be empty. Thus, $s=0$ as well.

For $r>0$, note that $S_1\subseteq X$ implies
\[
A_1\subseteq\left(\bigcup\limits_{i=1}^sB_i\right)\cup\left(\bigcup\limits_{j=1}^{m_1}A_{1,j}\right).
\]
By (2), it follows that $A_1$ is a subset of one of these balls. But since $S_1\neq\emptyset$, we must have $A_1\subseteq B_i$ for some $i\leq s$. By the same reasoning, $B_i\subseteq A_j$ for some $j\leq r$. This can be repeated until one of the balls appears twice, giving equality. So, renumbering for convenience, let us say that $A_1=B_1$.

Now we claim that $S_1=T_1$. To this end, note first that if $S_1$ has no holes, then $S_1=A_1=B_1$, and $T_1$ can have no holes either. To see this, suppose we have a hole $B_{1,1}$ of $T_1$. Since $B_{1,1}\subseteq A_1\subseteq X$, $B_{1,1}\subseteq \bigcup_{i\neq 1}T_i$. By (2), $B_{1,1}\subseteq B_i$ for some $i\neq 1$. Since $B_{1,1}\neq B_i$ by definition of swiss cheese decomposition, $B_{1,1}\subsetneq B_i$. If $B_i\subseteq B_1$, since $T_1\cap T_i=\emptyset$, $B_i$ would be covered by the holes of $T_1$ and $T_i$, contradicting unpackability. If $B_1\subseteq B_i$, then $T_1\cap T_i=\emptyset$ implies that $T_1$ is contained in the holes of $T_i$. In this case, unpackability gives that $B_1$ is contained in a hole of $T_i$, contradicting $B_{1,1}\subseteq T_i$. These contradictions imply that $T_1$ has no holes and $S_1=T_1$.

Otherwise, suppose $S_1$ has at least one hole $A_{1,1}$. Every element $x\in A_{1,1}$ is either
\begin{itemize}
\item 
not in $X$, in which case $x\in A_1=B_1$ implies $x\in B_{1,i}$ for some $i$; or,
\item
in $S_j$ for some $j\neq 1$. In this case, since $S_1\cap S_j=\emptyset$ and $A_j\neq A_{1,1}$, as before we must have $A_j\subsetneq A_{1,1}$.
\end{itemize}
Altogether,
\[
A_{1,1}\subseteq \left(\bigcup\limits_{i=1}^{n_i} B_{1,i}\right)\cup\left(\bigcup\limits_{A_j\subsetneq A_{1,1}}A_j\right)
\]
from which (2) implies $A_{1,1}\subseteq B_{1,i}$ for some $i$. The same argument applies to the other holes of $S_1$ and $T_1$, with the result that
\[
A_{1,1}\cup\ldots\cup A_{1,m_1}=B_{1,1}\cup\ldots\cup B_{1,n_1}
\]
and $S_1=T_1$.

Finally, the induction hypothesis applied to
\[
S_2\cup\ldots\cup S_r=T_2\cup\ldots\cup T_s
\]
finishes the proof.
\end{proof}

Combining (3) and (4) of the above theorem, it follows that in an unpackable directed family both the swiss cheeses in a decomposition of a constructible set $X$ and the wheels and holes of these swiss cheeses are uniquely determined by $X$. We thus obtain:

\begin{cor}
If $\cB$ is an unpackable directed family with $\fU\in\cB$, and $X\subseteq \fU$ is constructible, then $X=\Ch(\cS)$ for a unique finite set of balls $\cS$.
\end{cor}

\section{Forests and packable families}
\label{s3}

For this section, fix a set $\fU$ and $\cB\subseteq \cP(\fU)\setminus\{\emptyset\}$ directed, not necessarily unpackable. We again require $\fU\in\cB$. Consider $\Ch$ as defined before. For any constructible $X \subseteq \fU$, by Lemma \ref{l11} there exists some finite $\cS\subseteq\cB$ so that $X = \Ch(\cS)$. In this case, we will say that $\cS$ \emph{represents} $X$. 

However, since $\cB$ is potentially packable, we may have $\cS, \cT \subseteq \cB$ distinct but $\Ch(\cS) = \Ch(\cT)$. Nevertheless, in this section we describe a way to choose a canonical $\cS$ representing $X$.

Define, on the set of all finite forests, a quasi-ordering $\unlhd$ so that $\cS\unlhd\cT$ iff:
\begin{enumerate}[(i)]
\item
$|\cS|=|\cT|$ and for all $n$, $|\Lev_n(\cS)|=|\Lev_n(\cT)|$; or,
 \item
$\left|\cS\right|<\left|\cT\right|$; or,
 \item
$\left|\cS\right| = \left|\cT\right|$ and for some $n$ and all $i<n$, $\left|\Lev_i(\cS)\right| = \left|\Lev_i(\cT)\right|$ but $\left|\Lev_n(\cS)\right|>\left|\Lev_n(\cT)\right|$.
\end{enumerate}
So, roughly speaking, forests are ordered first by cardinality then by top-heaviness.  

\begin{figure}[h]
$\xymatrix@=4pt{
&&*=0{\bullet}\ar@{-}[dd]&*=0{\bullet}&*=0{\bullet}&&*=0{\bullet}\ar@{-}[dd]&*=0{\bullet}\ar@{-}[dd]&&&*=0{\bullet}\ar@{-}[ddl]\ar@{-}[ddr]&&*=0{\bullet}&&&*=0{\bullet}\ar@{-}[ddl]\ar@{-}[dd]\ar@{-}[ddr]&&&&*=0{\bullet}\ar@{-}[dl]\ar@{-}[dr]&&&&*=0{\bullet}\ar@{-}[d]&&& \\
\cdots&\lhd&&&&\lhd&&&\approx&&&&&\lhd&&&&\lhd&*=0{\bullet}\ar@{-}[d]&&*=0{\bullet}&\lhd&&*=0{\bullet}\ar@{-}[dl]\ar@{-}[dr]&&\lhd&\cdots \\
&&*=0{\bullet}&&&&*=0{\bullet}&*=0{\bullet}&&*=0{\bullet}&&*=0{\bullet}&&&*=0{\bullet}&*=0{\bullet}&*=0{\bullet}&&*=0{\bullet}&&&&*=0{\bullet}&&*=0{\bullet}&&
}$
\caption{Some forests of size $4$, quasi-ordered by $\unlhd$}
\end{figure}

We use this order to get a uniqueness of decomposition result:

\begin{thm}\label{Thm_CanonicalDecomposition}
Let $X\subseteq\fU$ be constructible and $\Ch(\cS)=\Ch(\cT)=X$. If $(\cS,\subseteq)$ and $(\cT,\subseteq)$ are both $\unlhd$-minimal among all representatives of $X$ in $\cB$, then $\cS= \cT$.
\end{thm}

\begin{proof}
Fix $\cS, \cT \subseteq \cB$ as in the hypothesis.  If $N = |\cS| = 0$, then $\cS = \cT = \emptyset$ and we are done.  We proceed now by induction on $N$.
 
Assume that $N > 0$.  We aim to prove
 \[
  \Lev_0(\cS) = \Lev_0(\cT).
 \]
By induction, this suffices.  Define, for $B\in\cS$, $R_S(B) = B\setminus\bigcup \Sub(B, \cS)$, and similarly $R_T(C)$ for $C\in\cT$. For $B$ and $C$ in even levels, these are the swiss cheeses comprising the decompositions generated by $\cS$ and $\cT$.  Note that since $R_S(B)=\emptyset$ or $R_T(C)=\emptyset$ clearly violates minimality of $N$, we can rule this possibility out.
 
So, consider a ball $B \in \Lev_0(\cS)$.  Since $R_S(B)\subseteq\Ch(\cS)=\Ch(\cT)$, there must exist $C\in \Lev_{2n}(\cT)$ for some $n$ such that $R_S(B) \cap R_T(C) \neq \emptyset$.  It follows that $B \cap C \neq \emptyset$. If $C \notin \Lev_0(\cT)$, then replace $C$ with the ball containing it in $\Lev_0(\cT)$.  In other words, we have shown that for any $B\in\Lev_0(\cS)$, there is $C\in\Lev_0(\cT)$ such that $B\cap C\neq\emptyset$. Since $\cB$ is directed, $B \subseteq C$ or $C \subseteq B$. Likewise, for any $C\in\Lev_0(\cT)$, there is $B\in\Lev_0(\cS)$ such that $B\subseteq C$ or $C\subseteq B$.
 
Now, if we did not have $\Lev_0(\cS)=\Lev_0(\cT)$, then by the above observation there would be $B\in\Lev_0(\cS)$ and $C\in\Lev_0(\cT)$ such that $B\subsetneq C$ or $C\subsetneq B$. Thus say, for instance, that we have found $B\subsetneq C$. Note that, as $B\in\Lev_0(\cS)$, there can be no $B'\in\cS$ for which $C\subseteq B'$.
 
Define 
\begin{align*}
\cS' =& \set{B' \in \Lev_0(\cS)}{B' \subseteq C\text{, but $B'\nsubseteq C'$ for any $C'\in\Sub(C,\cT)$}}\\
\cT' =& \set{C'\in\Sub(C,\cT)}{C'\cap B'=\emptyset\text{ for all $B'\in\cS'$}}.
\end{align*}
 
{\bf Claim 1:} $C\setminus \bigcup\cT'=\bigcup\cS'$.
 
First, if $x\in\bigcup\cS'$, then $x\in C$. If $C'\in\cT'$, then $C'\cap\left(\bigcup\cS'\right)=\emptyset$ gives $x\notin C'$. In other words, $x\in C\setminus\bigcup\cT'$.
 
Conversely, suppose $x\in C\setminus\bigcup\cT'$. If $x\notin X$, then since $x\in C$ we must have $x\in C'$ for some $C'\in\Sub(C,\cT)$. But $C'\notin\cT'$, so $C'\cap B'\neq\emptyset$ for some $B'\in\cS'$. Now $C'\subseteq B'$, whence $x\in\bigcup\cS'$. If on the other hand $x\in X$, then $x\in B'$ for some $B'\in\Lev_0(\cS)$. Suppose $B'\notin\cS'$. Since $B'\subsetneq C$ by choice of $C$, it follows that $B'\subseteq C'$ for some $C'\in\Sub(C,\cT)$. Since $x\in C'$ but not $\bigcup\cT'$, $C'\cap B''\neq\emptyset$ for some $B''\in\cS'$. Now $B'\subsetneq B''$ contradicts $B'\in\Lev_0(\cS)$, and the claim is proven.

Next, let
\begin{align*}
\cS^*=& \left(\cS\setminus\cS'\right)\cup\left\{C\right\}\cup\cT'\\
\cT^* =& \left(\cT\setminus\left(\left\{C\right\}\cup\cT'\right)\right)\cup\cS',
\end{align*} 
noting that $\cS'\subseteq\Lev_0(\cT^*)$, $C\in\Lev_0(\cS^*)$, and $\cT'\subseteq\Lev_1(\cS^*)$. 

{\bf Claim 2:} $\Ch(\cS^*)=X$.

For $\Ch(\cS^*)=X$, suppose first that $x\in X$. Then $x$ resides in a maximal chain
\begin{equation}
\label{e311}
x\in B_{2n}\subseteq B_{2n-1}\subseteq \ldots \subseteq B_0
\end{equation}
for some $B_i\in\Lev_i(\cS)$, $x\notin\bigcup\Sub(B_{2n},\cS)$. There are several cases to consider:
\begin{itemize}
\item
If $B_0\notin\cS'$, then either 
\begin{itemize}
\item
$B_0\cap C=\emptyset$ in which case $B_i\in\Lev_i(\cS^*)$ as well and $x\in\Ch(\cS^*)$; or
\item
$B_0\subseteq C$ and, since $B_0\notin\cS'$, $B_0\subseteq C'$ for some $C'\in\Sub(C,\cT)$. This $C'$ must be in $\cT'$ since $\bigcup\cS'=C\setminus\bigcup\cT'$. So in this case we have $B_i\in\Lev_{i+2}(\cS^*)$ in (\ref{e311}) and $x\in\Ch(\cS^*)$.
\end{itemize}
\item
If $B_0\in\cS'$, then $B_0\subseteq C$ but $B_0\cap C'=\emptyset$ for all $C'\in\cT'$. It follows that
\[
x\in B_{2n}\subseteq\ldots\subseteq B_1\subseteq C
\]
with $B_i\in\Lev_i(\cS^*)$, and $x\in\Ch(\cS^*)$.
\end{itemize}
This shows $X\subseteq\Ch(\cS^*)$.

For the converse, suppose again (\ref{e311}) but this time with $B_i\in\Lev_i(\cS^*)$, $x\notin\bigcup\Sub(B_{2n},\cS^*)$. If $B_0\in\cS\setminus\cS'$, then by choice of $C$, $B_0\cap C=\emptyset$ and $x\in\Ch(\cS)=X$. The other possibility is that $B_0=C$. Again there are several cases:
\begin{itemize}
\item
If $B_1\in\cT'$, then $B_2\in\Lev_0(\cS)$ and $x\in\Ch(\cS)$.
\item
If $B_1\in\cS\setminus\cT'$, then $B_1\nsubseteq C'$ for any $C'\in\cT'$ but $B_1\notin\cS'$ implies $B_1\notin\Lev_0(\cS)$. Since $B_1\in\Lev_1(\cS^*)$, it follows that $B_1\in\Lev_1(\cS)$ as well. So there is $B\in\Lev_0(\cS)$, $B_1\subseteq B$ and
\[
x\in B_{2n}\subseteq\ldots\subseteq B_1\subseteq B
\] 
with $B_i\in\Lev_i(\cS)$, giving $x\in\Ch(\cS)=X$.
\end{itemize}
We thus have shown $\Ch(\cS^*)=X$.

{\bf Claim 3:} $\Ch(\cT^*)=X$.

Suppose, similarly, $x\in X$ with a maximal chain
\begin{equation}
\label{e312}
x\in C_{2n}\subseteq \ldots \subseteq C_0,
\end{equation}
$C_i\in\Lev_i(\cT)$. So,
\begin{itemize}
\item
If $C_0 = C$, then either
\begin{itemize}
\item
$C_1\in\cT'$, in which case $C_1\cap B'=\emptyset$ for every $B'\in\cS'$, and $C_i\in\Lev_{i-2}(\cT^*)$ for any $i\ge 2$. This shows $x\in\Ch(\cT^*)$.
\item
$C_1\notin\cT'$, so that $x\in C\setminus\bigcup\cT'=\bigcup\cS'$. So $x\in B'$ for some $B'\in\cS'$, and the definition of $\cS'$ gives $C_1\subsetneq B'$. Thus the chain
\[
x\in C_{2n}\subseteq \ldots \subseteq C_1\subseteq B'
\]
with $C_i\in\Lev_i(\cT^*)$ gives $x\in\Ch(\cT^*)$.
\end{itemize}
\item
If $C_0\neq C$, then in (\ref{e312}), $C_i\in\Lev_i(\cT^*)$ as well, so again $x\in\Ch(\cT^*)$.
\end{itemize}

Conversely, suppose (\ref{e312}) but now with $C_i\in\Lev_i(\cT^*)$, $x\notin\bigcup\Sub(C_{2n},\cT^*)$. If $C_0\in\cS'$, then $x\in C\setminus\bigcup\cT'$. Since $C_1\subseteq C_0$, $C_1\cap C'=\emptyset$ for all $C'\in\cT'$. Thus we have in $\cT$ the (maximal) chain
\[
x\in C_{2n}\subseteq\ldots\subseteq C_1\subseteq C
\]
with $C_i\in\Lev_i(\cT)$.

On the other hand, suppose $C_0\in\cT\setminus\cS'$. Then since $C_0\in\Lev_0(\cT^*)$, $C_0\cap B'=\emptyset$ for all $B'\in\cS'$. It follows that either $C_0\cap C=\emptyset$, or $C_0\subseteq C'$ for some $C'\in\cT'$. In the first case, the chain $C_i$ is the same in $\cT$ and $x\in\Ch(\cT)$. In the second case, since $C_0\notin\cT'$, the chain becomes
\[
x\in C_{2n}\subseteq\ldots\subseteq C_0\subseteq C'\subseteq C
\]
with $C_i\in\Lev_{i+2}(\cT)$. This shows again that $x\in\Ch(\cT)$, and the claim is proven.

Finally, depending on the relative sizes of $|\cS'|$ and $|\cT'|$, we derive a contradiction to $\unlhd$-minimality of $\cS$ and $\cT$ as follows:
\begin{itemize}
\item
If $|\cS'|<|\cT'|+1$, then $\cT^*$ represents $X$ with strictly fewer balls than $\cT$.
\item
If $|\cS'|>|\cT'|+1$, then $\cS^*$ represents $X$ with strictly fewer balls than $\cS$.
\item
If $|\cS'|=|\cT'|+1=1$, say $\cS'=\left\{B\right\}$, then $C\setminus\bigcup\cT'=C=B$ contradicts our choice of $C$.
\item
If $|\cS'|=|\cT'|+1\ge 2$, then $\cS^*$ represents $X$ with $N$ balls, but with $\left|\Lev_0(\cS^*)\right|>\left|\Lev_0(\mathcal{S})\right|$.
\end{itemize}

The contradiction gives $\Lev_0(\cS)=\Lev_0(\cT)$, and the result follows. 
\end{proof}

\section{VC-minimality}
\label{s4}

While the previous sections relied purely on the combinatorial properties of directed families, the questions originate in logic with the notion of VC-minimality.

A complete theory $T$ is \emph{VC-minimal} if there is a family of formulas $\Psi$ such that the set
\[
\set{\psi(x,\bar{a})}{\psi(x,\bar{y})\in\Psi, \bar{a}\in\fM^{|\bar{y}|}}
\]
of instances of $\Psi$ form a directed family for every $\fM\models T$; and such that every definable subset of $\fM$ is a boolean combination of instances of $\Psi$. Here we do not distinguish between $\psi(x,\bar{a})$ and the set it defines in a model $\fM$ of $T$. Note moreover that the length of the tuples $\bar{y}$ may vary with $\psi(x,\bar{y})$, but $x$ is exclusively a single variable. $\Psi$ is called a \emph{generating family}. A formula $\delta(x;\bar{z})$ in $T$ is \emph{directed} if for any $\fM\models T$, the set of instances of $\delta$ in $\fM$ forms a directed family.

The terminology around directed families carries over naturally to VC-minimal theories, with the definable subsets of $\fM$ being the constructible sets. As with directed families, for our purposes it will be most convenient to stick to the convention that $x=x\in\Psi$ but $x\neq x\notin\Psi$, i.e. the whole universe is always a ball, the empty set is never a ball.

Likewise, Theorems \ref{t12} and \ref{Thm_CanonicalDecomposition} can be applied immediately to the family of balls generated by any model of a VC-minimal theory.

\subsection{Uniform definability of levels}
In this subsection, we observe that the levels of a canonical decomposition as in \ref{Thm_CanonicalDecomposition} are uniformly definable. This fact will be useful in type counting arguments in VC-minimal theories (see \cite{vince}).

Fix a formula $\phi(x; \bar{y})$ in a VC-minimal theory.  By compactness, there exists a single directed $\delta(x; \bar{z})$ and $N < \omega$ so that all instances of $\phi$ are a boolean combination of at most $N$ instances of $\delta(x; \bar{z})$. (More precisely, compactness gives us finitely many $\psi \in \Psi$, then we can use coding tricks to combine them into a single directed $\delta$.) As we will only be working with instances of $\phi$, we disregard $\Psi$ and work instead in the directed family of instances of $\delta$.

There are only finitely many forests of size at most $N$; call this set $\mathcal{F}_N$.  For each $F \in \mathcal{F}_N$, let $\psi_F(\bar{y})$ denote the formula which says that there exists $\bar{z}_f$ for $f \in F$ so that $\phi(x; \bar{y})$ is equivalent to $\Ch\left( \set{\delta(x; \bar{z}_f)}{f \in F } \right)$ and $\delta(x; \bar{z}_f)$ are ordered so that $f \mapsto \delta(x; \bar{z}_f)$ is an isomorphism of forests.  

Finally, for any $n < N$, let $\gamma_n(x; \bar{y})$ denote the formula that says, for the $\unlhd$-least $F \in \mathcal{F}_N$ such that $\psi_F(\bar{y})$ holds, there exists witnesses $\bar{z}_f$ for $f \in F$ as above such that $\delta(x; \bar{z}_f)$ holds for some $f \in \Lev_n(F)$.  Thus, for any $\bar{b}$, $\gamma_n(x; \bar{b})$ holds if and only if $x$ appears in the $n$th level of a $\unlhd$-minimal decomposition $\set{\delta(x; \bar{c}_f)}{f \in F }$ of $\phi(x;\bar{b})$.  However, by Theorem \ref{Thm_CanonicalDecomposition}, the set $\set{\delta(x; \bar{c}_f)}{f \in F }$ is unique up to $T$-equivalence.  Therefore, $\gamma_n(x; \bar{b})$ holds if and only if $x$ is in the $n$th level of the $\unlhd$-minimal decomposition.  

We summarize in the following theorem:

\begin{thm}\label{Thm_CanonicalLevels}
 If $\phi(x; \bar{y})$ is any formula in a VC-minimal theory $T$, there exists a directed $\delta(x;\bar{z})$, $N_\phi < \omega$ and $\gamma_{\phi,n}(x; \bar{y})$ for all $n < N_\phi$ such that:
 \begin{enumerate}[(i)]
  \item For all $\bar{b}$, $\gamma_{\phi,n}(x;\bar{b})$ is $T$-equivalent to a disjoint union of (at most $N_\phi$) instances of $\delta$.
  \item $\phi(x;\bar{y})$ is $T$-equivalent to 
\[
\bigvee\limits_{n < N_\phi}\left( \gamma_{\phi,2n}(x; \bar{y}) \wedge \neg \gamma_{\phi,2n+1}(x; \bar{y})\right).
\]
 \end{enumerate}
\end{thm}

\subsection{VC-minimality and imaginaries}

In this subsection, we outline an application which extends the analogy to Holly's work with algebraically closed valued fields \cite{hol2}. The main observation is that the canonical representation of definable sets from \ref{Thm_CanonicalDecomposition} leads to a one-dimensional elimination of imaginaries.

A theory is said to \emph{eliminate imaginaries} if, for every model $\fM$, $n\in\bbN$, and definable set $X\subseteq \fM^n$, there is a formula $\phi(\bar{x};\bar{y})$ and tuple $\bar{a}\in\fM^{|\bar{y}|}$ such that for all $\bar{b}\in\fM^{|\bar{y}|}$, $\phi(\bar{x};\bar{b})$ defines $X$ iff $\bar{b}=\bar{a}$. In this case, $\bar{a}$ is called a \emph{code} (or \emph{canonical parameter}) of $X$. The existence of codes allows one, in a sense, to treat definable sets as elements of the model. See \cite{poi} for further discussion.

It should be noted that it is always possible to expand a model $\fM$ to a (usually multi-sorted) structure $\fM^{\text{eq}}$ which eliminates imaginaries by explicitly adding to the language a code for every definable set. This suffices for many applications, but in other situations one may find a better understanding of the definable sets in a structure by finding a way to expand the language to achieve elimination of imaginaries in a more efficient, or natural, way.

The notion of codes also specializes naturally to definable sets of a certain dimension. To this end, say a theory has \emph{$n$-prototypes} if there is a family $\Phi=\left\{\phi(\bar{x};\bar{y})\right\}$ with $|\bar{x}|=n$ such that for every model $\fM$ and every definable set $X\subseteq\fM^n$, there is exactly one $\phi\in\Phi$ and $\bar{a}\in\fM^{|\bar{y}|}$ such that $\phi(\bar{x};\bar{a})$ defines $X$. Holly proves in \cite{hol2} that a theory eliminates imaginaries iff it has $n$-prototypes for every $n\ge 1$. It is also clear from the proof that a theory has $1$-prototypes iff every definable subset of a model (in one variable) has a code.

Returning to VC-minimal theories, the definable sets in more than one variable are not yet well understood. The favorite example of algebraically closed valued fields indicates that the situation can be quite complex (see for instance \cite{hhm}). However, on the question of $1$-prototypes, the work of the preceding sections does the trick. We present this in two forms.

Suppose $T$ is VC-minimal, with generating family $\Psi$. Expand the language of $T$ to add to any model $\fM\models T$ a new sort consisting of the finite sets of balls. Add also, for each $\langle\psi_1(x;\bar{y}_1),\ldots,\psi_n(x;\bar{y}_n)\rangle\in\Psi^n$, a new function symbol from the main sort to the new sort,
\[
f_{\left[\psi_1,\ldots,\psi_n\right]}\left(\bar{y}_1,\dots,\bar{y}_n\right):\langle\bar{a}_1,\ldots,\bar{a}_n\rangle\mapsto\left\{B_1,\ldots,B_n\right\}
\]
taking the tuple of parameters $\langle\bar{a}_1,\ldots,\bar{a}_n\rangle$ to the set of balls defined by $\psi_i(x;\bar{a}_i)=B_i$. Let $T^*$ be the theory of a model of $T$ expanded in this way.

\begin{thm}
\label{t51}
If $T$ is VC-minimal, then $T^*$ has $1$-prototypes.
\end{thm}

\begin{proof}
We show only that every definable set $X\subseteq\fM$ has a code. Assume $X\neq\fM$; that $\fM$ has a code is obvious. By Theorem \ref{Thm_CanonicalDecomposition}, there is a finite set $\cS=\left\{B_1,\ldots,B_n\right\}$ of balls so that $\Ch(\cS)=X$ and if $\Ch(\cT)=X$ and $\left(\cT,\subseteq\right)\cong\left(\cS,\subseteq\right)$ (as forests), then $\cS=\cT$. 

Now, for $i\le n$ let $B_i=\psi_i(x;\bar{a}_i)$, $\psi\in\Psi$. Let $\phi(x;\mathcal{Y})$ be the formula stating that
\[
\forall\langle\bar{y}_1,\ldots,\bar{y}_n\rangle\in  f_{\left[\psi_1,\ldots,\psi_n\right]}^{-1}(\mathcal{Y})
\left(
\begin{array}{l}
\left(\left\{\psi_i(x;\bar{y}_i)\right\}_{i\le n},\subseteq\right)\cong\left(\cS,\subseteq\right)\  \rightarrow  \\
\hspace{1in}x\in\Ch\left(\left\{\psi_i(x;\bar{y}_i)\right\}_{i\le n}\right)
\end{array}
\right).
\]
Since $X\neq\fM$, if $\phi(x;\cT)$ defines $X$ then $(\cT,\subseteq)\cong(\cS,\subseteq)$, and hence $\cT=\cS$. So $\cS$ is in fact a code for $X$.
\end{proof}

This language is complicated somewhat by the need to allow for all finite sets of balls. This is necessary, as it is not generally possible, for example, to distinguish $\langle B_1, B_2\rangle$ from $\langle B_2, B_1\rangle$ in terms of the definable set represented by this pair of balls. This phenomenon is commonplace enough to earn its own terminology.

A theory \emph{weakly eliminates imaginaries} if for every model $\fM$, $n\in\bbN$, and definable set $X\subseteq\fM^n$, there is a formula $\phi(\bar{x};\bar{y})$ and nonempty finite set $A\subseteq\fM^{|\bar{y}|}$ such that for all $\bar{b}\in\fM^{|\bar{y}|}$, $\phi(\bar{x};\bar{b})$ defines $X$ iff $\bar{b}\in A$. Analogously, a theory has \emph{weak $1$-prototypes} if there is a family $\Phi$ such that for every model $\fM$ and every definable set $X\subseteq\fM$, there is exactly one $\phi(x;\bar{y})\in\Phi$ and finitely many $\bar{a}\in\fM^{|\bar{y}|}$ such that $\phi(x;\bar{a})$ defines $X$ .

Now from a VC-minimal $T$ construct an expanded theory $T^\circ$ by adding a new sort consisting of the balls, as well as for each $\psi\in\Psi$ a new function symbol
\[
f_\psi:\bar{a}\mapsto B=\psi(x,\bar{a}).
\] 
As the codes from Theorem \ref{t51} depended only on the set of balls $\left\{B_1,\ldots,B_n\right\}$, we may replace them with an ordered tuple of balls $\langle B_1,\ldots,B_n\rangle$ at the expense of allowing as many as $n!$ potential codes rather than only one. We thus obtain

\begin{cor}
\label{c42}
If $T$ is VC-minimal, then $T^\circ$ has weak $1$-prototypes.
\end{cor}

\subsection{Quasi-VC-minimality}

A theory $T$ is called \emph{quasi-VC-minimal} if there exists a family $\Psi$ such that the set of all instances of formulas from $\Psi$ is directed and every parameter-definable subset of a model of $T$ is a boolean combination of instances of $\Psi$ and $\emptyset$-definable sets. An example of a quasi-VC-minimal but not VC-minimal theory is Presburger arithmetic, $\Th\left(\bbZ;+,<\right)$. See \cite{fg} for details.

We outline how the above results can be adapted to apply to quasi-VC-minimal theories. As the main differences in the proofs are notational annoyances, these are omitted.

Given $\fM\models T$ and $\emptyset$-definable $Q\subseteq\fM$, the restriction of the balls to $Q$, $\set{B\cap Q}{\text{$B$ a ball}}$ is again a directed family. Thus a boolean combination of balls intersected with $Q$  admits a canonical, $\unlhd$-minimal swiss cheese decomposition as in Theorem \ref{Thm_CanonicalDecomposition}. Here, the balls themselves are not uniquely defined, but only their intersection with $Q$. For a finite set $\cS=\{B_i\}_i$ of balls, write $\cS\cap Q=\{B_i\cap Q\}_i$.

Now, given a formula $\phi(x;\bar{y})$, there are formulas $\theta_1(x),\ldots,\theta_k(x)$ over $\emptyset$ such that every instance of $\phi$ is a boolean combination of balls and $\theta_1,\ldots,\theta_k$. For $e:\{1,\ldots,k\}\rightarrow\{0,1\}$, write
\[
\theta^e(x)=\bigwedge_{i=1}^k\theta_i(x)^{e(i)}.
\]
Then as in \ref{l11} it is proved that every instance of $\phi$ can be written as 
\[
\bigvee_e\left(\theta^e(x)\wedge \sigma_e(x,\bar{a}_e)\right)
\] 
where $\sigma_e$ defines a swiss cheese decomposition. Again, the balls in this swiss cheese decomposition are not uniquely determined, but by \ref{Thm_CanonicalDecomposition} they can be chosen so that their intersections with $\theta^e$ are:

\begin{thm}
If $X\subseteq\fM$ is definable, then there are pairwise disjoint $\emptyset$-definable $Q_1,\ldots,Q_k\subseteq\fM$ partitioning $\fM$ and, for each $i\leq k$ a finite set of balls $\cS_i$ such that
\begin{enumerate}[(i)]
\item
$X\cap Q_i = \Ch\left(\cS_i\cap Q_i\right)$, and
\item
for any $\cT$, if also
$X\cap Q_i = \Ch\left(\cT\cap Q_i\right)$
and $\left(\cT\cap Q_i,\subseteq\right)\cong\left(\cS_i\cap Q_i,\subseteq\right)$,
then $\cT\cap Q_i=\cS_i\cap Q_i$.
\end{enumerate}
\end{thm}

Finally, for quasi-VC-minimal $T$, let $T^\sharp$ be the theory obtained by adding 
\begin{itemize}
\item
a new sort consisting of the intersections of balls with $\emptyset$-definable sets, and
\item
for each $\psi(x;\bar{y})\in\Psi$ and each formula $\theta$ over $\emptyset$, a new function symbol
\[
f_{\left[\theta,\psi\right]}(\bar{y}): \bar{a}\mapsto \theta(x)\wedge\phi(x;\bar{a}).
\]
\end{itemize}
We then obtain as in \ref{c42}:

\begin{cor}
$T^\sharp$ has weak 1-prototypes.
\end{cor}

\end{document}